\documentclass[12pt]{amsart}
\usepackage{amsmath,amssymb,bm,amsthm,mathrsfs,color}
\newtheorem{theorem}{Theorem}[section]
\newtheorem{lemma}[theorem]{Lemma}
\newtheorem{corollary}[theorem]{Corollary}
\newtheorem{example}[theorem]{Example}
\newtheorem{remark}[theorem]{Remark}
\numberwithin{equation}{section}
\numberwithin{table}{section}
\allowdisplaybreaks[4]
\begin{document}
\def\R{\mathbb{R}}
\def\bK{\mathbb{K}}
\def\argmin{\mathop{\rm argmin}}
\def\fA{\mathscr{A}\hspace{-2pt}}
\def\fK{\mathscr{K}\hspace{-3pt}}
\def\TG{T_{\scriptscriptstyle G}}
\def\EG{E_{\scriptscriptstyle G}}
\def\muG{\mu_{\scriptscriptstyle \bar{G}}}
\def\cT{\mathcal{T}}
\def\tbar{|\!|\!|}
\def\LT{{L_2(I)}}
\title[FEMs for Optimal Control Problems with Derivative Constraints]
{Finite element methods for one dimensional
  elliptic distributed optimal control problems with pointwise constraints on the derivative of the state}
\author{S.C. Brenner}
\address{Susanne C. Brenner, Department of Mathematics and Center for
Computation and Technology, Louisiana State University, Baton Rouge, LA 70803, USA}
\email{brenner@math.lsu.edu}
\author{L.-Y. Sung}
\address{Li-yeng Sung,
 Department of Mathematics and Center for Computation and Technology,
 Louisiana State University, Baton Rouge, LA 70803, USA}
\email{sung@math.lsu.edu}
\author{W. Wollner}
\address{Winnifried Wollner, Department of Mathematics, Technische Universit\"at Darmstadt, Germany}
\email{wollner@mathematik.tu-darmstadt.de}
\thanks{The work of the first and second authors was supported in part
 by the National Science Foundation under Grant Nos.
 DMS-16-20273 and DMS-19-13035.}
\date{\today}
\begin{abstract}
  We investigate  $C^1$ finite element methods for  one dimensional
  elliptic distributed optimal control problems with pointwise constraints on the
  derivative of the state formulated as fourth order variational inequalities for the state variable.
   For the problem with Dirichlet boundary conditions, we use an existing
   $H^{\frac52-\epsilon}$ regularity result for the
   optimal state to derive
   $O(h^{\frac12-\epsilon})$ convergence for the approximation of the optimal state in the $H^2$ norm.
   For the problem with mixed Dirichlet and Neumann boundary conditions, we show that the optimal state belongs to $H^3$ under appropriate assumptions on the data
    and obtain $O(h)$ convergence for the approximation of the optimal state
   in the $H^2$ norm.
\end{abstract}
\maketitle
\section{Introduction}\label{sec:Introduction}
 Let $I$ be the interval $(-1,1)$ and the
 function $J:L_2(I)\times L_2(I)\longrightarrow \R$ be defined by
\begin{equation}\label{eq:JDef}
  J(y,u)=\frac12 \big[\|y-y_d\|_\LT^2+\beta\|u\|_\LT^2\big],
\end{equation}
 where  $y_d\in L_2(I)$ and $\beta$ is a positive constant.
\par
 The optimal control problem is to
\begin{equation}\label{eq:OCP}
  \text{find}\quad(\bar y,\bar u)=\argmin_{(y,u)\in \bK}J(y,u),
\end{equation}
 where $(y,u)\in H^2(I)\times L_2(I)$ belongs to $\bK$ if and only if
\begin{alignat}{3}
  -y''&=u+f&\qquad&\text{on $I$},\label{eq:PDEConstraint}\\
  y'&\leq \psi &\qquad&\text{on $I$},\label{eq:DerivativeConstraint}
\end{alignat}
 together with the following boundary conditions for $y$:
\begin{subequations}\label{subeq:BoundaryConditions}
\begin{align}
  y(-1)=y(1)&=0, \label{eq:Dirichlet}\\
\intertext{or}
y(-1)=y'(1)&=0.\label{eq:Mixed}
\end{align}
\end{subequations}
\begin{remark}\label{rem:Notation}\rm
  Throughout this paper we will follow standard notation for function spaces and norms that can be
   found, for example, in \cite{Ciarlet:1978:FEM,BScott:2008:FEM,ADAMS:2003:Sobolev}.
\end{remark}
\par
 For the problem with the Dirichlet boundary conditions \eqref{eq:Dirichlet},
 we assume that
\begin{equation}\label{eq:DirichletData}
  f\in H^{\frac12-\epsilon}(I), \;\psi\in H^{\frac32-\epsilon}(I) \quad\forall\,\epsilon>0 \quad\text{and}
   \quad \int_I \psi\,dx>0.
\end{equation}
\par
 For the problem with the mixed boundary conditions \eqref{eq:Mixed}, we assume that
\begin{equation}\label{eq:MixedData}
  f\in H^1(I),\;\psi\in H^2(I) \quad \text{and} \quad \psi(1)\geq 0.
\end{equation}
\begin{remark}\label{rem:Boring}\rm
 In the case of Dirichlet boundary conditions, clearly we need $\int_I\psi\,dx\geq0$ since $\int_I y' dx=0$
    and $y'\leq\psi$.  However $\int_I \psi\,dx=0$ implies
   $\int_I(y'-\psi)dx=0$,
  which together with $y'\leq\psi$ leads to $y'=\psi$.  Hence in this case $\bK$ is a singleton
  and the optimal control problem becomes trivial.
\end{remark}
\par
 The optimal control problem  with
 the Dirichlet boundary conditions \eqref{eq:Dirichlet}
  is a  one dimensional analog
 of the optimal control problems considered in
  \cite{CB:1988:Gradient,CF:1993:Gradient,DGH:2009:Gradient,OW:2011:Gradient,Wollner:2012:Gradient}
  on smooth or convex domains.
 In \cite{CB:1988:Gradient,CF:1993:Gradient}, first order optimality conditions were derived for a
 semilinear elliptic optimization problem with pointwise gradient constraints on smooth domains,
 where the solution of the state equation is in $W^{1,\infty}$ for any feasible control.
 These results were extended to non-smooth domains in \cite{Wollner:2012:Gradient}.
 On the other hand higher dimensional analogs of the optimal control problem with
 the mixed boundary conditions \eqref{eq:Mixed} are absent from the literature.
\par
 Finite element error analysis for the problem with the Dirichlet boundary conditions was first carried out in
 \cite{DGH:2009:Gradient} by a mixed formulation of the elliptic equation
 and a variational discretization of the control, and in \cite{OW:2011:Gradient} by a standard
 $H^1$-conforming discretization with a possible non-variational control discretization.
\par
 The goal of this paper is to show that it is also possible to solve the one dimensional
 optimal control problem with either boundary conditions
 as a fourth order variation inequality for the state variable by a $C^1$ finite element method.
 We note that such an approach has been carried out for elliptic distributed optimal control
 problems with pointwise state constraints in, for example, the papers
 \cite{LGY:2009:Control,BSZ:2013:OptimalControl,BDS:2014:PUMOC,BOPPSS:2016:OC3D,BGPS:2018:Morley,
 BGS:2018:Nonconvex,BSZ:2019:Neumann}.  The analysis in this paper extends the general framework in
 \cite{BSung:2017:State} to the one dimensional problem defined by \eqref{eq:JDef}--\eqref{subeq:BoundaryConditions}.
\par
 The rest of the paper is organized as follows.  We collect information on the optimal control problem
 in Section~\ref{sec:Continuous}.  The construction and analysis of the discrete problem are treated in Section~\ref{sec:Discrete}, followed by numerical results in
 Section~\ref{sec:Numerics}.  We end with some concluding remarks in Section~\ref{sec:Conclusion}.
 The appendices contain derivations of the Karush-Kuhn-Tucker conditions that appear
  in Section~\ref{sec:Continuous}.
\par
 Throughout the paper we will use $C$ (with or without subscript) to denote
 a generic positive constant independent of the mesh sizes.
\section{The Continuous Problem}\label{sec:Continuous}
 Let the space $V$ be defined by
\begin{subequations}\label{subeq:V}
\begin{alignat}{3}
 V&=\{v\in H^2(I):\, v(-1)=v(1)=0\}&\qquad&\text{for the boundary conditions \eqref{eq:Dirichlet}},
 \label{eq:VDirichlet}
\intertext{and}
 V&=\{v\in H^2(I):\,v(-1)=v'(1)=0\}&\qquad&\text{for the boundary conditions \eqref{eq:Mixed}}.
 \label{eq:VMixed}
\end{alignat}
\end{subequations}
\par
  The minimization problem defined by \eqref{eq:JDef}--\eqref{subeq:BoundaryConditions} can be reformulated
 as the  following problem that only involves $y$:
\begin{equation}\label{eq:ROCP}
  \text{Find}\quad \bar y=\argmin_{y\in K}\frac12\big[\|y-y_d\|_\LT^2+\beta\|y''+f\|_\LT^2\big],
\end{equation}
 where
\begin{equation}\label{eq:KDef}
  K=\{y\in V:\,y'\leq \psi \;\text{on}\;I\}.
\end{equation}
\par
 Note that the closed convex subset $K$ of the Hilbert space
 $V$ is nonempty for either boundary
 conditions.  In the case of the Dirichlet boundary conditions, the function
  $y(x)=\int_{-1}^ x (\psi(t)-\delta)dt$ belongs to $K$ if we take $\delta$ to be $\frac12\int_I\psi\,dx \,(>0)$.
   Similarly, in the case
 of the mixed boundary conditions, the function $y(x)=\int_{-1}^x[\psi(t)-\delta\sin[(\pi/4)(1+t)]dt$
 belongs to $K$ if we take $\delta$ to be $\psi(1)\,(\geq0)$.
\par
  According to the standard theory \cite{ET:1999:Convex},
  there is a unique solution $\bar y$ of \eqref{eq:ROCP}--\eqref{eq:KDef}
  characterized by the fourth order variational inequality
\begin{equation}\label{eq:VI}
    \int_I (\bar y-y_d)(y-\bar y)dx+ \beta\int_I (\bar y''+f)(y''-\bar y'')dx\geq 0
    \qquad\forall\,y\in K.
\end{equation}
\par
 We can express \eqref{eq:VI} in the form of
\begin{equation}\label{eq:NewVI}
  a(\bar y,y-\bar y)\geq \int_I y_d(y-\bar y)dx-\beta\int_I f(y''-\bar y'')dx
  \qquad\forall\,y\in K,
\end{equation}
 where
\begin{equation}\label{eq:aDef}
  a(y,z)=\beta\int_I y''z''dx+\int_I yz\,dx.
\end{equation}
\subsection{The Karush-Kuhn-Tucker Conditions}\label{subsec:KKT}
 The solution of \eqref{eq:VI} is characterized by the following conditions:
\begin{alignat}{3}
  \int_I (\bar y-y_d)z\,dx+\beta\int_I (\bar y''+f)z''\,dx+\int_{[-1,1]} z'd\mu&=0
  &\qquad&\forall\,z\in V,\label{eq:KKT1}\\
    \int_{[-1,1]} (\bar y'-\psi)d\mu&=0,\label{eq:KKT2}
\end{alignat}
 where
\begin{equation}\label{eq:KKT3}
  \text{$\mu$ is a nonnegative finite Borel measure on $[-1,1]$.}
\end{equation}
\par
{  Note that \eqref{eq:KKT2} is equivalent to the statement that
 $\mu$ is supported on the active set
\begin{equation}\label{eq:ActiveSet}
 \fA=\{x\in [-1,1]:\,\bar y'(x)=\psi(x)\}
\end{equation}
 for the derivative constraint \eqref{eq:DerivativeConstraint}.
\par
 We can also express \eqref{eq:KKT1} as
\begin{equation}\label{eq:NewKKT1}
  a(\bar y,z)-\int_I y_d z\,dx+\beta\int_I fz''dx=-\int_{[-1,1]}z'd\mu \qquad\forall\,z\in V.
\end{equation}
}
\par
 The Karush-Kuhn-Tucker
 (KKT) conditions \eqref{eq:KKT1}--\eqref{eq:KKT3} can be derived from
  the general theory on Lagrange multipliers that can be found, for example,
   in  \cite{Luenberger:1969:Optimization,IK:2008:Lagrange}.  For the simple
 one dimensional problem here, they can also be derived directly (cf.  Appendix~\ref{appendix:Dirichlet}
 for the Dirichlet boundary conditions and Appendix~\ref{appendix:Mixed} for
   the mixed boundary conditions).

\begin{remark}\label{rem:Mixedmu}\rm
  In the case of the mixed boundary conditions, additional
   information on the structure of
    $\mu$ (cf. {\eqref{eq:muStructure}}) 
     is obtained
     in Appendix~\ref{appendix:Mixed}.
\end{remark}
%
\subsection{Dirichlet Boundary Conditions}\label{sec:Dirichlet}
 We will use \eqref{eq:KKT1} to obtain additional regularity for $\bar y$
 that matches the regularity result in \cite{OW:2011:Gradient}.
 The following lemmas are useful for this purpose.
\begin{lemma}\label{eq:Interpolation1}
   We have
\begin{equation}\label{eq:fRegularity}
 \int_I fv'\,dx\leq C_\epsilon|f|_{H^{\frac12-\epsilon}(I)}
  |v|_{H^{\frac12+\epsilon}(I)}\qquad\forall\,v\in H^1(I)
  \quad\text{and}\quad\epsilon\in (0,1/2).
\end{equation}
\end{lemma}
\begin{proof}
   Observe that
\begin{alignat}{3}
 \int_I gv'dx&\leq \|g\|_{L_2(I)}|v|_{H^1(I)}&\qquad&\forall\,v\in H^1(I) \label{eq:Interp1}
 \intertext{if $g\in L_2(I)$, and}
 \int_I gv'dx&\leq |g|_{H^1(I)}\|v\|_{L_2(I)}&\qquad&\forall\,v\in H^1(I)\label{eq:Interp2}
\end{alignat}
 if $g\in H^1_0(I)$.
\par
 Recall that $f\in H^{\frac12-\epsilon}(I)$ by the assumption in \eqref{eq:DirichletData}.
 The estimate \eqref{eq:fRegularity} follows from
 \eqref{eq:Interp1}, \eqref{eq:Interp2} and bilinear interpolation
 (cf. \cite[Theorem~4.4.1]{BL:1976:Interpolation}),
  together with the following
  interpolations of Sobolev spaces (cf. \cite[Sections~1.9 and 1.11]{LM:1972:NHBVPI}):
\begin{align*}
  [L_2(I),H^1_0(I)]_{\frac12-\epsilon}=H^{\frac12-\epsilon}_0(I)=H^{\frac12-\epsilon}(I)
  \quad\text{and}\quad
  [H^1(I),L_2(I)]_{\frac12-\epsilon}=H^{\frac12+\epsilon}(I).
\end{align*}
\end{proof}
\par
 Note that the map $z\rightarrow z''$ is an isomorphism between $V$
 (given by \eqref{eq:VDirichlet}) and $L_2(I)$.
 Therefore, by the Riesz representation theorem, for any $\ell\in V'$ we can define $p\in L_2(I)$ by
\begin{equation}\label{eq:Abstractp}
 \int_I pz''\,dx=\ell(z) \qquad\forall\,z\in V.
\end{equation}
\begin{lemma}\label{lem:UltraWeak}
  Given any $s\in [0,1]$,  the function $p$ defined by \eqref{eq:Abstractp}
   belongs to $H^{1-s}(I)$ provided that
\begin{equation}\label{eq:UltraWeakEst}
 \ell(z)\leq C|z|_{H^{1+s}(I)}  \qquad\forall\,z\in H^{1+s}(I).
\end{equation}
\end{lemma}
\begin{proof}
 On  one hand, if $\ell\in(H^2(I))'$, we have
\begin{equation}\label{eq:UW1}
  \|p\|_{L_2(I)}\leq \|\ell\|_{(H^2(I))'}.
\end{equation}
\par
 On the other hand, if $\ell\in (H^1(I))'$, then the solution $p$ of \eqref{eq:Abstractp}
 can also be defined by the conditions that  $p\in H^1_0(I)$ and
\begin{equation*}
  \int_I p'q'dx=-\ell(q)\qquad\forall\,q\in H^1_0(I).
\end{equation*}
 Hence in this case we have
\begin{equation}\label{eq:UW2}
  |p|_{H^1(I)}\leq \|\ell\|_{(H^1(I))'}.
\end{equation}
\par
 The estimate \eqref{eq:UltraWeakEst} follows from \eqref{eq:UW1}, \eqref{eq:UW2} and
  the following interpolations of  Sobolev spaces (cf. \cite[Sections~1.6 and 1.9]{LM:1972:NHBVPI}):
\begin{align*}
  [L_2(I),H^1(I)]_{1-s}&=H^{1-s}(I)
\intertext{and}
[(H^2(I))',(H^1(I))']_{1-s}&=([H^1(I),H^2(I)]_s)'
    =(H^{1+s}(I))'.
\end{align*}
\end{proof}
\goodbreak
\begin{theorem}\label{thm:DirichletStateRegularity}
  The solution $\bar y$ of \eqref{eq:VI} belongs to $H^{\frac52-\epsilon}(I)$ for all
  $\epsilon\in (0,1/2)$.
\end{theorem}
\begin{proof}
 Note that, by the Sobolev inequality \cite{ADAMS:2003:Sobolev},
\begin{equation}\label{eq:muRegularity}
  \int_I v\,d\mu\leq C_\epsilon |v|_{H^{\frac12+\epsilon}(I)}\qquad\forall\,v\in H^1(I)
  \quad\text{and}\quad\epsilon\in (0,1/2).
\end{equation}
\par
 Let $p\in L_2(I)$ be defined by
\begin{equation}\label{eq:UltraWeak}
  \beta\int_I pz''\,dx=\int_I (y_d-\bar y)z\,dx-\beta\int_I fz''\,dx-\int_{[-1,1]} z'd\mu\qquad\forall\,z\in V,
\end{equation}
 where $V$ is given by \eqref{eq:VDirichlet}.
 It follows from \eqref{eq:fRegularity}, \eqref{eq:muRegularity},
 \eqref{eq:UltraWeak} and
 Lemma~\ref{lem:UltraWeak} (with $s=\frac12+\epsilon$) that
\begin{equation}\label{eq:DirichletAdjointRegularity}
     \text{$p$ belongs to $H^{\frac12-\epsilon}(I)$ for all $\epsilon\in(0,1/2)$}.
\end{equation}
\par
 Comparing \eqref{eq:KKT1} and \eqref{eq:UltraWeak}, we see that
\begin{equation*}
  \int_I \bar y'' z''dx=\int_I pz''\,dx\qquad\forall\,z\in V
\end{equation*}
 and hence $\bar y''=p$, which together with \eqref{eq:DirichletAdjointRegularity} concludes the proof.
\end{proof}
\begin{corollary}\label{cor:DirichletControlRegularity}
 We have $\bar u=-\bar y''-f\in H^{\frac12-\epsilon}(I)$ for all $\epsilon\in (0,1/2)$.
\end{corollary}
%
\begin{example}\label{example:Dirichlet}\rm
 We take $\beta=\psi=1$ and the exact solution
\begin{equation}\label{eq:DirichletExactSolution}
  \bar y(x)=\begin{cases}
    -\frac12(x+1)+\frac12(x+1)^3+\frac{1}{12}(1-x^2)^3&\qquad -1< x\leq 0\\[6pt]
    -\frac12(x-1)+\frac12(x-1)^3+\frac{1}{12}(1-x^2)^3&\qquad \hspace{10pt}0\leq x<1
  \end{cases}.
\end{equation}
 It follows from a direct calculation that
\begin{equation*}
  \bar y'(x)=\begin{cases} -\frac12+\frac32(x+1)^2-\frac12x(1-x^2)^2&\qquad -1< x\leq 0\\[6pt]
           -\frac12+\frac32(x-1)^2-\frac12x(1-x^2)^2&\qquad \hspace{10pt}0\leq x<1
           \end{cases},
\end{equation*}
 and
\begin{equation*}
   \bar y''(x)=\begin{cases} 3(x+1)-\frac12(1-6x^2+5x^4)&\qquad -1< x< 0\\[6pt]
            3(x-1)-\frac12(1-6x^2+5x^4)&\qquad \hspace{10pt}0< x<1
           \end{cases}.
\end{equation*}
 It is straightforward to check that $\bar y$ belongs to $K$, $\fA=\{0\}$, and for $z\in V$,
\begin{align}\label{eq:Example1}
  \int_I \bar y''z''dx&=\int_{-1}^0 3(x+1)z''dx+\int_0^1 3(x-1)z''dx
   -\frac12\int_I(1-6x^2+5x^4)z''dx\\
        &=6z'(0)+\int_I gz\,dx,\notag
\end{align}
 where
 $$g(x)=6(1-5x^2).$$
\par
 Now we take
\begin{equation*}
  f(x)=\begin{cases} 7(x^2-1) &\qquad -1<x<0\\[6pt]
                  0&\qquad\hspace{10pt} 0<x<1
                  \end{cases}
\end{equation*}
 so that $f\in H^{\frac12-\epsilon}(I)$ for all $\epsilon>0$ and
\begin{equation}\label{eq:Example2}
  \int_I fz''dx=7\int_{-1}^0 (x^2-1)z''dx
          =-7z'(0)+14\int_{-1}^0 z\,dx\qquad\forall\,z\in V.
\end{equation}
\par
 Putting \eqref{eq:Example1} and \eqref{eq:Example2} together we have
\begin{equation}\label{eq:Example3}
    -\int_I (14\chi_{(-1,0)}+g)z\,dx+ \int_I (\bar y''+f)z''dx+z'(0)=0 \qquad\forall\,z\in V,
\end{equation}
 where
    $$\chi_S(x)=\begin{cases}
      1&\qquad\text{if $x\in S$}\\[4pt]
      0&\qquad\text{if $x\notin S$}
    \end{cases}$$
  is the characteristic function of the set $S$,
 and the KKT conditions \eqref{eq:KKT1}--\eqref{eq:KKT3} are satisfied (with
 $\mu$ being the Dirac point measure at the origin)
 if we choose
\begin{equation}\label{eq:yd}
  y_d=\bar y+14\chi_{(-1,0)}+g,
\end{equation}
\end{example}
\begin{remark}\label{rem:Sharp}\rm
  It follows from Example~\ref{example:Dirichlet} that the regularities of $\bar y$ and $\bar u$ stated in Theorem~\ref{thm:DirichletStateRegularity} and   Corollary~\ref{cor:DirichletControlRegularity}
   are sharp under the assumptions on the data in \eqref{eq:DirichletData}.
\end{remark}
\subsection{Mixed Boundary Conditions}\label{subsec:Mixed}
 In this case the nonnegative Borel measure $\mu$ on $[-1,1]$ satisfies
 (cf. \eqref{eq:nuFormula}--\eqref{eq:muFormula})
\begin{equation}\label{eq:muStructure}
  d\mu = \beta[\rho\,dx+\gamma d\delta_{-1}],
\end{equation}
 where $\rho\in L_2(I)$ is nonnegative, $\gamma$ is a nonnegative number and $\delta_{-1}$ is the
 Dirac point measure at $-1$.
\begin{theorem}\label{thm:MixedStateRegularity}
 The solution $\bar y$ of \eqref{eq:VI} belongs to $H^3(I)$.
\end{theorem}
\begin{proof}  Recall that $f\in H^1(I)$ by the assumption in \eqref{eq:MixedData}.
 After substituting \eqref{eq:muStructure} into \eqref{eq:KKT1} and
carrying out integration by parts,
 we have
\begin{equation}\label{eq:MixedKKT1}
  \beta\int_I  \bar y''z''\,dx=\int_I (y_d-\bar y) z\,dx+\beta\int_I (f'-\rho)z'dx+\beta[f(-1)-\gamma]z'(-1)\qquad\forall\,z\in V,
\end{equation}
 where $V$ is given by \eqref{eq:VMixed}.
\par
 Let $H^1(I;1)=\{v\in H^1(I):\,v(1)=0\}$ and $p\in H^1(I;1)$ be defined by
\begin{equation}\label{eq:MixedAdjointEq}
   \int_I p'q'dx=-\int_I \Phi q\,dx+\int_I (f'-\rho)qdx+[f(-1)-\gamma]q(-1)
 \qquad\forall\,q\in H^1(I;1),
\end{equation}
 where $\Phi\in H^1(I;1)$ is defined by
\begin{equation}\label{eq:PhiMixed}
   \beta\Phi'=y_d-\bar y.
\end{equation}
\par
 Note that \eqref{eq:MixedAdjointEq} is the weak form of the two-point boundary value problem
\begin{equation*}
  -p''=-\Phi+f'-\rho\quad\text{in $I$}\quad\text{and}\quad
  p'(-1)=\gamma-f(-1),\;p(1)=0,
\end{equation*}
 and hence we can conclude from elliptic regularity that
\begin{equation}\label{eq:MixedAdjointRegularity}
 p\in H^2(I).
\end{equation}
\par
 Finally \eqref{eq:MixedKKT1}--\eqref{eq:PhiMixed} imply
\begin{equation*}
 \int_I \bar y''z''dx=\int_I p'z''\,dx\qquad\forall\,z\in V
\end{equation*}
 and hence $\bar y''=p'$ because the map $z\rightarrow z''$ is also an isomorphism between
 $V$ (defined by \eqref{eq:Mixed}) and $L_2(I)$.
    The theorem then follows from
 \eqref{eq:MixedAdjointRegularity}.
\end{proof}
\begin{corollary}\label{cor:MixedControlRegularity}
  We have $\bar u=-\bar y''-f\in H^1(I)$.
\end{corollary}
\begin{example}\label{example:Mixed}\rm
 We take $\beta=\psi=1$, $f=0$ and
 the exact solution is given by
\begin{equation}\label{eq:MixedExactSolution}
  \bar y(x)=\int_{-1}^x p(t)dt,
\end{equation}
 where
\begin{equation}\label{eq:pFormula}
  p(x)=\begin{cases}
    1&\qquad -1< x\leq \frac13\\[4pt]
     \sin \big[\frac{\pi}{4}(9x-1)\big]&\qquad\hspace{8pt}\frac13\leq x< 1
  \end{cases}.
\end{equation}
\par
 We have $\fA=[-1,1/3]$,
 $p\in H^2(I)$,
\begin{equation}\label{eq:pProperties}
   p_+''(1/3)=-({9\pi}/4)^2 \quad\text{and}\quad p(1)=p''(1)=0.
\end{equation}
\par
 If we choose the function $\Phi$ by
\begin{equation}\label{eq:PhiFormula}
  \Phi(x)=\begin{cases}
    -({9\pi}/4)^2 &\qquad -1\leq x\leq \frac13\\[4pt]
     p''(x)&\qquad\hspace{8pt} \frac13\leq x\leq 1
  \end{cases},
\end{equation}
 then $\Phi\in H^1(I;1)$ by \eqref{eq:pProperties} and \eqref{eq:PhiFormula},  and
\begin{equation}\label{eq:1NeumannExampleKKT1}
 \int_I p'q'dx= -\int_I \Phi q\,dx-\int_{-1}^\frac13  ({9\pi}/4)^2q\,dx \qquad\forall\,q\in H^1(I;1).
\end{equation}
 Therefore \eqref{eq:MixedAdjointEq} is valid if we take
\begin{equation}\label{eq:NeumannExamplemu}
  \rho= ({9\pi}/4)^2\chi_{[-1,1/3]} \quad\text{and} \quad \gamma=0.
\end{equation}
\par
 Finally we define $y_d$ according to \eqref{eq:PhiMixed}
 so that
\begin{equation}\label{eq:NeumannExampleyd}
  y_d(x)=\begin{cases}
     \bar y(x)&\qquad -1<x<\frac13\\[4pt]
     \bar y(x)+p'''(x)&\qquad\hspace{8pt} \frac13< x< 1
  \end{cases}.
\end{equation}
\par
 Putting \eqref{eq:MixedExactSolution} and
 \eqref{eq:1NeumannExampleKKT1}--\eqref{eq:NeumannExampleyd} together, we see
 that the KKT conditions \eqref{eq:KKT1}--\eqref{eq:KKT3} are valid provided we define
 the Borel measure $\mu$ by
\begin{equation*}
  d\mu=({9\pi}/4)^2\chi_{[-1,1/3]}dx.
\end{equation*}
\end{example}
%

\section{The Discrete Problem}\label{sec:Discrete}
 Let $\cT_h$ be a quasi-uniform partition of $I$ and $V_h\subset V$
  be the cubic Hermite finite
 element space \cite{Ciarlet:1978:FEM} associated with $\cT_h$.  The discrete problem is to
\begin{equation}\label{eq:DP}
  \text{find}\quad \bar y_h=\argmin_{y_h\in K_h}\frac12
  \big[\|y_h-y_d\|_\LT^2+\beta\|y_h''+f\|_\LT^2\big],
\end{equation}
 where
\begin{equation}\label{eq:KhDef}
  K_h=\{y\in V_h:\; P_hy'\leq P_h\psi \quad\text{on $[-1,1]$}\},
\end{equation}
 and $P_h$ is the nodal interpolation operator for the $P_1$ finite
 element space \cite{Ciarlet:1978:FEM,BScott:2008:FEM}
 associated with $\cT_h$.
 In other words the derivative constraint \eqref{eq:DerivativeConstraint}
 is only imposed at the grid points.
\par
 The nodal interpolation operator from $C^1(\bar I)$ onto $V_h$ will be denoted by $\Pi_h$.
 Note that
\begin{equation}\label{eq:C2D}
  \Pi_h y\in K_h \qquad\forall\,y\in K.
\end{equation}
 In particular, the closed convex set $K_h$ is nonempty.
\par
 The minimization problem \eqref{eq:DP}--\eqref{eq:KhDef}
 has a unique solution characterized by the discrete variational inequality
\begin{equation*}
    \int_I (\bar y_h-y_d)(y_h-\bar y_h)dx+ \beta\int_I (\bar y_h''+f)(y_h''-\bar y_h'')dx\geq 0
    \qquad\forall\,y_h\in K_h,
\end{equation*}
 which can also be written as
\begin{equation}\label{eq:DVI}
  a(\bar y_h,y_h-\bar y_h)\geq \int_I y_d(y_h-\bar y_h)dx-
     \beta\int_I f(y_h''-\bar y_h'')dx \qquad\forall\,y_h\in K_h.
\end{equation}
\par
 We begin the error analysis by recalling some properties of
 $P_h$ and $\Pi_h$.
\par
  For $0\leq s\leq 1\leq t\leq 2$,
 we have an error estimate
\begin{equation}\label{eq:PhEst}
   \|\zeta-P_h\zeta\|_{H^s(I)}\leq Ch^{t-s}|\zeta|_{H^t(I)} \qquad\forall\,\zeta\in H^t(I)
\end{equation}
 that follows from standard  error estimates for $P_h$ (cf. \cite{Ciarlet:1978:FEM,BScott:2008:FEM})
 and interpolation between Sobolev spaces \cite{ADAMS:2003:Sobolev}.
\par
 For $0\leq s\leq 1$ and $2\leq t\leq 4$, we also have the estimates
\begin{alignat}{3}
  \|\zeta-\Pi_h\zeta\|_\LT+h^2|\zeta-\Pi_h\zeta|_{H^2(I)}& \leq C h^t|\zeta|_{H^t(I)}
    &\qquad&\forall\,\zeta\in H^s(I),\label{eq:PihErrorEstimate1}\\
    |\zeta-\Pi_h\zeta|_{H^{1+s}(I)}&\leq C h^{t-s-1}
    |\zeta|_{H^t(I)}    &\qquad&\forall\,\zeta\in H^s(I)\label{eq:PihErrorEstimate2},
\end{alignat}
 that follow from standard error estimates
  for $\Pi_h$ (cf. \cite{Ciarlet:1978:FEM,BScott:2008:FEM}) and interpolation between Sobolev spaces.
\subsection{An Intermediate Error Estimate}\label{subsec:Abstract}
 Let the energy norm $\|\cdot\|_a$ be defined by
\begin{equation}\label{eq:EnergyNorm}
\|v\|_a^2=a(v,v)=\|v\|_{L_2(I)}^2+\beta|v|_{H^2(I)}^2.
\end{equation}
 We have, by a Poincar\' e$-$Friedrichs inequality \cite{Necas:2012:Direct},
\begin{equation}\label{eq:Equivalence}
 C_1\|v\|_a\leq \|v\|_{H^2(I)}\leq C_2\|v\|_a \qquad\forall\,v\in V.
\end{equation}
\par
  Observe that  \eqref{eq:DVI}, \eqref{eq:EnergyNorm}
   and the Cauchy-Schwarz inequality  imply
\begin{align}\label{eq:Preliminary}
  \|\bar y-\bar y_h\|_a^2&=a(\bar y-\bar y_h,\bar y-y_h)+
      a(\bar y-\bar y_h,y_h-\bar y_h)\notag\\
      &\leq \frac12\|\bar y-\bar y_h\|_a^2+\frac12\|\bar y-y_h\|_a^2+a(\bar y,y_h-\bar y_h)\\
            &\hspace{60pt}
             -\int_I y_d(y_h-\bar y_h)dx+
     \beta\int_I f(y_h''-\bar y_h'')dx\qquad\forall\,y_h\in K_h,\notag
\end{align}
 and we have, by \eqref{eq:KKT2}--\eqref{eq:NewKKT1} and \eqref{eq:KhDef},
\begin{align}\label{eq:KKTRelation}
  &a(\bar y,y_h-\bar y_h)-\int_I y_d(y_h-\bar y_h)dx+
     \beta\int_I f(y_h''-\bar y_h'')dx\notag\\
     &\hspace{40pt}=\int_{[-1,1]} (\bar y_h'-y_h')d\mu\notag\\
     &\hspace{40pt}=\int_{[-1,1]} (\bar y_h'-P_h\bar y_h')d\mu+
     \int_{[-1,1]} (P_h\bar y_h'-P_h\psi)d\mu+\int_{[-1,1]} (P_h\psi-\psi)d\mu\\
        &\hspace{90pt}+\int_{[-1,1]} (\psi-\bar y')d\mu+
         \int_{[-1,1]} (\bar y'-y_h')d\mu,\notag\\
       &\hspace{40pt}\leq \int_{[-1,1]} (\bar y_h'-P_h\bar y_h')d\mu+\int_{[-1,1]} (P_h\psi-\psi)d\mu
       +\int_{[-1,1]} (\bar y'-y_h')d\mu
       \notag
\end{align}
 for all $y_h\in K_h$.
\par
 It follows from \eqref{eq:Preliminary} and \eqref{eq:KKTRelation} that
\begin{align}\label{eq:AbstractErrorEstimate}
 \|\bar y-\bar y_h\|_a^2&\leq 2\Big[\int_{[-1,1]} (\bar y_h'-P_h\bar y_h')d\mu
   +\int_{[-1,1]}(P_h\psi-\psi)d\mu\Big]\\
   &\hspace{50pt} +\inf_{y_h\in K_h}\Big(\|\bar y-y_h\|_a^2+
   2\int_{[-1,1]} (\bar y'-y_h')d\mu\Big).\notag
\end{align}
\subsection{Dirichlet Boundary Conditions}\label{subsec:DirichletError}
  The following estimates will allow us to produce concrete error estimates from
  \eqref{eq:AbstractErrorEstimate}.  First of all, we have
\begin{align}\label{eq:DirichletPhEstimate}
   \int_{[-1,1]} (\bar y_h'-P_h\bar y_h')d\mu&
   =\int_{[-1,1]} \big[(\bar y_h'-\bar y')-P_h(\bar y_h'-\bar y')\big]d\mu+
      \int_{[-1,1]} (\bar y'-P_h\bar y')d\mu\\
     &\leq C_\epsilon \Big( h^{\frac12-\epsilon}\|\bar y-\bar y_h\|_a
     +h^{1-\epsilon}|y|_{H^{\frac52-\epsilon}(I)}\Big)&\forall\,\epsilon>0\notag
 \end{align}
 by \eqref{eq:muRegularity},  Theorem~\ref{thm:DirichletStateRegularity},
 \eqref{eq:PhEst} and \eqref{eq:Equivalence}; secondly
\begin{equation}\label{eq:DirichletpsiError}
 \int_{[-1,1]} (P_h\psi-\psi)d\mu\leq C_\epsilon h^{1-\epsilon} |\psi|_{H^{\frac32-\epsilon}(I)} \qquad\forall\epsilon>0
\end{equation}
 by the assumption on $\psi$ in \eqref{eq:DirichletData} and \eqref{eq:PhEst}.  Finally,
 in view of Theorem~\ref{thm:DirichletStateRegularity}, \eqref{eq:muRegularity},
  \eqref{eq:PihErrorEstimate1}--\eqref{eq:PihErrorEstimate2} and
   \eqref{eq:Equivalence}, we also have
\begin{equation}\label{eq:DirichletPihEstimate}
  \|\bar y-\Pi_h \bar y\|_a^2+2\int_{[-1,1]} \big[\bar y'-(\Pi_h\bar y)'\big]d\mu
  \leq C_\epsilon h^{1-\epsilon}
  \qquad\forall\,\epsilon>0.
 \end{equation}
\par
 Putting \eqref{eq:C2D}, \eqref{eq:AbstractErrorEstimate}--\eqref{eq:DirichletPihEstimate}
 and Young's inequality together,
  we arrive at the estimate
\begin{align}\label{eq:DirichletConcretErrorEstimate}
  \|\bar y-\bar y_h\|_a\leq C_\epsilon h^{\frac12-\epsilon}
\end{align}
 that is valid for any $\epsilon>0$,
 which in turn  implies the following
 result, where $\bar u_h=-\bar y_h''-f$ is the approximation for $\bar u=-\bar y''-f$.
\begin{theorem}\label{thm:DirichletOCPErrors}
  Under the assumptions on the data in \eqref{eq:DirichletData}, we have
\begin{equation*}
  |\bar y-\bar y_h|_{H^1(I)}+\|\bar u-\bar u_h\|_\LT\leq C_\epsilon h^{\frac12-\epsilon}
    \qquad \forall\,\epsilon>0.
\end{equation*}
\end{theorem}
\begin{remark}\label{rem:HOD}\rm
  Numerical results in Section~\ref{sec:Numerics} indicate that $|\bar y-\bar y_h|_{H^1(I)}$
  is of higher order.
\end{remark}
\subsection{Mixed Boundary Conditions}\label{subsec:MixedError}
 In this case we have
\begin{align}\label{eq:MixedPhEstimate}
   \int_{[-1,1]} (\bar y_h'-P_h\bar y_h')d\mu&=\beta\Big[\int_I (\bar y_h'-P_h\bar y_h')\rho\,dx+
      \gamma(\bar y_h'-P_h\bar y_h')(-1)\Big]\notag\\
      &=\beta\Big[\int_I \big[(\bar y_h'-\bar y')-P_h(\bar y_h'-\bar y')\big]\rho\,dx+
       \int_I (\bar y'-P_h\bar y')\rho\,dx\Big]\\
     &\leq C  \Big( h\|\bar y-\bar y_h\|_a
     +h^2|\bar y|_{H^3(I)}\Big)\notag
 \end{align}
 by \eqref{eq:muStructure},  Theorem~\ref{thm:MixedStateRegularity},
 \eqref{eq:PhEst} and \eqref{eq:Equivalence};
\begin{align}\label{eq:MixedpsiEerror}
  \int_{[-1,1]} (P_h\psi-\psi)d\mu
          &=\beta\int_I (P_h\psi-\psi)\rho dx\leq Ch^2
\end{align}
 by the assumption on $\psi$ in \eqref{eq:MixedData}, \eqref{eq:muStructure} and \eqref{eq:PhEst}; and
\begin{equation}\label{eq:MixedPihEstimate}
  \|\bar y-\Pi_h \bar y\|_a^2+2\int_{[-1,1]} \big[\bar y'-(\Pi_h\bar y)'\big]d\mu
  \leq Ch^2
 \end{equation}
 by \eqref{eq:muStructure}, Theorem~\ref{thm:MixedStateRegularity},
 \eqref{eq:PihErrorEstimate1}, \eqref{eq:PihErrorEstimate2} and \eqref{eq:Equivalence}.
\par
 Combining \eqref{eq:AbstractErrorEstimate} and \eqref{eq:MixedPhEstimate}--\eqref{eq:MixedPihEstimate}
 with Young's inequality,
 we find
\begin{equation}\label{eq:MixedStateError}
  \|\bar y-\bar y_h\|_a\leq Ch,
\end{equation}
 which immediately implies the following result,
  where $\bar u_h=-\bar y_h''-f$ is the approximation
 for $\bar u=-\bar y''-f$.
\begin{theorem}\label{thm:MixedOCPErrors}
  Under the assumptions on the data in \eqref{eq:MixedData}, we have
\begin{equation*}
  |\bar y-\bar y_h|_{H^1(I)}+\|\bar u-\bar u_h\|_\LT\leq Ch.
\end{equation*}
\end{theorem}
\begin{remark}\label{rem:HOM}\rm
 Numerical results in Section~\ref{sec:Numerics} again indicate that $|\bar y-\bar y_h|_{H^1(I)}$
  is of higher order.
\end{remark}
\section{Numerical Results}\label{sec:Numerics}
 In the first experiment, we solved the problem in Example~\ref{example:Dirichlet} on a uniform
 mesh with dyadic grid points.
 The errors of $\bar y_h$ in various norms are reported in Table~\ref{table:DUniform}.  We observed
 $O(h^2)$ convergence in $|\cdot|_{H^2(I)}$ and higher convergence in the lower order norms.
 This phenomenon can be justified as follows.
\par
 Note that for this example
 the first term on the right-hand side of \eqref{eq:AbstractErrorEstimate}
  vanishes because $\mu$ is supported at the origin which is one of the grid points where
  $\bar y_h$ (resp. $\psi$) and $P_h\bar y_h$ (resp., $P_h\psi$)
  assume identical values.
    The remaining term on the right-hand side
  of \eqref{eq:AbstractErrorEstimate} is bounded by
 \begin{align*}
   \|\bar y-(\Pi_h\bar y)\|_a^2+2\int_I\big[\bar y'-(\Pi_h\bar y)'\big]d\mu= \|\bar y-(\Pi_h\bar y)\|_a^2\leq Ch^4,
 \end{align*}
 where we have used  the estimate \eqref{eq:PihErrorEstimate1},
  with $I$ replaced by the intervals $(-1,0)$ and $(0,1)$,
  the norm equivalence \eqref{eq:Equivalence},
   and the  fact that $\bar y$ defined by \eqref{eq:DirichletExactSolution}
   is a sextic polynomial on each of these intervals.
\begin{table}[ht]
\begin{tabular}{|c|c|c|c|l|}
\hline
\rule{0pt}{2.5ex}
DOFs&$\|\bar y-\bar y_h\|_{L_2(I)}$ &$\|\bar y-\bar y_h\|_{L_\infty(I)}$
 &$|\bar y-\bar y_h|_{H^1(I)}$&$|\bar y-\bar y_h|_{H^2(I)}$\\[2pt]
\hline
\rule{0pt}{2.5ex}
\hspace{-4pt}$2^1$  &1.082369 \hspace{-2pt}e$-$01& 1.545433 \hspace{-2pt}e$-$01& 3.788872 \hspace{-2pt}e$-$01& 2.178934 \hspace{-2pt}e$+$00\\
 $2^2$& 5.972336 \hspace{-2pt}e$-$03 &7.142850 \hspace{-2pt}e$-$03 &2.452678 \hspace{-2pt}e$-$02 & 7.191076 \hspace{-2pt}e$-$01\\
$2^3$& 1.223603 \hspace{-2pt}e$-$03 &1.806781 \hspace{-2pt}e$-$03& 8.520509 \hspace{-2pt}e$-$03 &1.114423 \hspace{-2pt}e$-$01\\
 $2^4$& 8.653379 \hspace{-2pt}e$-$05& 1.732075 \hspace{-2pt}e$-$04 &1.200903 \hspace{-2pt}e$-$03 &3.118910 \hspace{-2pt}e$-$02\\
$2^5$& 5.561252 \hspace{-2pt}e$-$06 &1.295847 \hspace{-2pt}e$-$05 &1.542654 \hspace{-2pt}e$-$04 &8.001098 \hspace{-2pt}e$-$03\\
 $2^6$ & 3.508709 \hspace{-2pt}e$-$07 & 8.804766 \hspace{-2pt}e$-$07 &1.929895 \hspace{-2pt}e$-$05 &2.012955 \hspace{-2pt}e$-$03\\
$2^7$ & 2.199861 \hspace{-2pt}e$-$08 & 5.729676 \hspace{-2pt}e$-$08 &2.303966 \hspace{-2pt}e$-$06 &5.040206 \hspace{-2pt}e$-$04\\
\hline
\end{tabular}
\par\medskip
\caption{Numerical results for Example~\ref{example:Dirichlet} on meshes with dyadic grid points}
\label{table:DUniform}
\end{table}
\par
 In the second experiment we solved the problem
 in Example~\ref{example:Dirichlet} on slightly perturbed meshes where
 the origin is no longer a grid point.  The errors are reported in
   Table~\ref{table:DPerturbed}.  We observed $O(h^{0.5})$ convergence in the $|\cdot|_{H^2(I)}$
 (which agrees with Theorem~\ref{thm:DirichletOCPErrors}) and $O(h)$ convergence in the lower
 order norms.
\begin{table}[ht]
\begin{tabular}{|c|c|c|c|c|}
\hline
\rule{0pt}{2.5ex}
DOFs&$\|\bar y-\bar y_h\|_{L_2(I)}$ &$\|\bar y-\bar y_h\|_{L_\infty(I)}$
&$|\bar y-\bar y_h|_{H^1(I)}$&$|\bar y-\bar y_h|_{H^2(I)}$\\[2pt]
\hline
\rule{0pt}{2.5ex}
\hspace{-4pt}$2$+$2^1$ &5.972336 \hspace{-2pt}e$-$03 &7.142850 \hspace{-2pt}e$-$03 &2.452678 \hspace{-2pt}e$-$02 &1.910760 \hspace{-2pt}e$-$01\\
$2$+$2^2$ & 3.045281 \hspace{-2pt}e$-$02 &3.279329 \hspace{-2pt}e$-$02  &1.188082 \hspace{-2pt}e$-$01 &1.285638 \hspace{-2pt}e$+$00\\
$2$+$2^3$&3.187355 \hspace{-2pt}e$-$02 &3.182310 \hspace{-2pt}e$-$02& 1.071850 \hspace{-2pt}e$-$01 &1.022401 \hspace{-2pt}e$+$00\\
$2$+$2^4$&3.216705 \hspace{-2pt}e$-$02& 3.175715 \hspace{-2pt}e$-$02 &1.048464 \hspace{-2pt}e$-$01& 8.070390 \hspace{-2pt}e$-$01\\
$2$+$2^5$& 3.220153 \hspace{-2pt}e$-$02& 3.175558 \hspace{-2pt}e$-$02 &1.044763 \hspace{-2pt}e$-$01 &6.496040 \hspace{-2pt}e$-$01\\
$2$+$2^6$&1.814346 \hspace{-2pt}e$-$02& 2.074403 \hspace{-2pt}e$-$02 &5.740999 \hspace{-2pt}e$-$02 &4.408863 \hspace{-2pt}e$-$01\\
$2$+$2^7$ & 9.754613 \hspace{-2pt}e$-$03 &1.167762 \hspace{-2pt}e$-$02& 2.983716 \hspace{-2pt}e$-$02 &3.016101 \hspace{-2pt}e$-$01\\
\hline
\end{tabular}
\par\medskip
\caption{Numerical results for Example~\ref{example:Dirichlet} on meshes where
 $0$ is not a grid point}
\label{table:DPerturbed}
\end{table}
\par
 In the third experiment, we solved the problem in Example~\ref{example:Mixed}
 on a uniform mesh with dyadic grid points.
   We observed $O(h)$ convergence in $|\cdot|_{H^2(I)}$ from the results in Table~\ref{table:MDyadic}
 (which agrees with Theorem~\ref{thm:MixedOCPErrors}) and $O(h^2)$ convergence
 in the lower order norms.
\begin{table}[ht]
\begin{tabular}{|c|c|c|c|c|}
\hline
\rule{0pt}{2.5ex}
DOFs&$\|\bar y-\bar y_h\|_{L_2(I)}$ &$\|\bar y-\bar y_h\|_{L_\infty(I)}$
&$|\bar y-\bar y_h|_{H^1(I)}$&$|\bar y-\bar y_h|_{H^2(I)}$\\[2pt]
\hline
\rule{0pt}{2.5ex}
\hspace{-4pt}$1+2^2$&1.406813 \hspace{-2pt}e$+$01&1.658318 \hspace{-2pt}e$+$01&1.269278 \hspace{-2pt}e$+$01&2.070271 \hspace{-2pt}e$+$01\\
$1+2^3$&4.654073 \hspace{-2pt}e$+$00&4.618639 \hspace{-2pt}e$+$00&4.221134 \hspace{-2pt}e$+$00&1.379991 \hspace{-2pt}e$+$01\\
$1+2^4$&1.574605 \hspace{-2pt}e$+$00&1.683229 \hspace{-2pt}e$+$00&1.376788 \hspace{-2pt}e$+$00&8.047102 \hspace{-2pt}e$+$00\\
$1+2^5$&3.745106 \hspace{-2pt}e$-$01&3.781562 \hspace{-2pt}e$-$01&3.252880 \hspace{-2pt}e$-$01&4.073631 \hspace{-2pt}e$+$00\\
$1+2^6$&9.856747 \hspace{-2pt}e$-$02&1.022258 \hspace{-2pt}e$-$01&8.574934 \hspace{-2pt}e$-$02&2.081469 \hspace{-2pt}e$+$00\\
$1+2^7$&2.378457 \hspace{-2pt}e$-$02&2.368760 \hspace{-2pt}e$-$02&2.075267 \hspace{-2pt}e$-$02&1.037836 \hspace{-2pt}e$+$00\\
$1+2^8$&5.802109 \hspace{-2pt}e$-$03&5.661900 \hspace{-2pt}e$-$03&5.218542 \hspace{-2pt}e$-$03&5.212004 \hspace{-2pt}e$-$01\\
\hline
\end{tabular}
\par\medskip
\caption{Numerical results for Example~\ref{example:Mixed} on meshes with dyadic grid points}
\label{table:MDyadic}
\end{table}
\par
 In the final experiment, we solved the problem in Example~\ref{example:Mixed} by a uniform
 mesh that includes $1/3$ as a grid point.  The errors are reported in
 Table~\ref{table:MNonDyadic}.
 We observed similar convergence behavior as the dyadic case, but
 the magnitude of the errors is smaller.  This can be justified by the observation that
 the term (cf. \eqref{eq:MixedPhEstimate})
     $$\int_I (\bar y-P_h\bar y')\rho\,dx=\int_0^\frac13 (\bar y-P_h\bar y')\rho\,dx=0$$
 because $\bar y(x)=1+x$ on the active set $\fA=[-1,1/3]$ and $1/3$ is a grid point.
 On the other hand the corresponding integral is nonzero for dyadic meshes.
\begin{table}[ht]
\begin{tabular}{|c|c|c|c|c|}
\hline
\rule{0pt}{2.5ex}
DOFs&$\|\bar y-\bar y_h\|_{L_2(I)}$ &$\|\bar y-\bar y_h\|_{L_\infty(I)}$
&$|\bar y-\bar y_h|_{H^1(I)}$&$|\bar y-\bar y_h|_{H^2(I)}$\\[2pt]
\hline
\rule{0pt}{2.5ex}
\hspace{-4pt}$1+3\cdot2^1$&2.448013 \hspace{-2pt}e$+$00&2.343224 \hspace{-2pt}e$+$00 &2.236575 \hspace{-2pt}e$+$00&1.082726 \hspace{-2pt}e$+$01\\
$1+3\cdot2^2$&6.406496 \hspace{-2pt}e$-$01&6.095607 \hspace{-2pt}e$-$01&5.795513 \hspace{-2pt}e$-$01&5.541353 \hspace{-2pt}e$+$00\\
$1+3\cdot2^3$&1.616111 \hspace{-2pt}e$-$01&1.539557 \hspace{-2pt}e$-$01&1.461718 \hspace{-2pt}e$-$01&2.778978 \hspace{-2pt}e$+$00\\
$1+3\cdot2^4$&4.025578 \hspace{-2pt}e$-$02&3.858795 \hspace{-2pt}e$-$02&3.665436 \hspace{-2pt}e$-$02&1.390198 \hspace{-2pt}e$+$00\\
$1+3\cdot2^5$&9.822613 \hspace{-2pt}e$-$03&9.653193 \hspace{-2pt}e$-$03&9.268709 \hspace{-2pt}e$-$03&6.951994 \hspace{-2pt}e$-$01\\
$1+3\cdot2^6$&2.233582 \hspace{-2pt}e$-$03&2.413687 \hspace{-2pt}e$-$03& 2.657435 \hspace{-2pt}e$-$03&3.476583 \hspace{-2pt}e$-$01\\
\hline
\end{tabular}
\par\medskip
\caption{Numerical results for Example~\ref{example:Mixed}
on uniform meshes where $1/3$ is a grid point}
\label{table:MNonDyadic}
\end{table}
%
\section{Concluding Remarks}\label{sec:Conclusion}
 We have demonstrated in this paper that the convergence analysis developed in \cite{BSung:2017:State}
 can be adopted to elliptic distributed optimal control problems with pointwise constraints on the derivatives of the state, at least in a simple one dimensional setting.
\par
 The results in this paper can be extended to two-sided constraints of the form
 $
  \psi_1\leq y'\leq \psi_2,
 $
 where $\psi_i$ and $\psi_2$ are sufficiently regular and $\psi_1<0<\psi_2$ on $I$.  In particular, they
 are valid for the constraints defined by $|y'|\leq 1$.
\par
 It would be interesting to find out if the results in this paper can be extended to higher dimensions.
{ We note that the higher dimensional analogs of the variational inequality for the derivative (cf.
 \eqref{eq:NewNeumannVI}) lead to  obstacle problems for the vector Laplacian.
 Such obstacle problems are of independent interest and appear to be open.  }
\appendix
\section{KKT Conditions for the Dirichlet Boundary Conditions}\label{appendix:Dirichlet}
 First we note that
\begin{equation}\label{eq:DirichletActiveSet}
  \fA\neq [-1,1]
\end{equation}
 since $\int_I y'dx=0$ and $\int_I\psi\,dx>0$, and also
\begin{equation}\label{eq:Quotient}
  \{y':\,y\in V\}=\{v\in H^1(I):\,\int_I v\,dx=0\}=H^1(I)/\R.
\end{equation}
\par
 Let $\fK=\{v\in H^1(I)/\R:\,v\leq \psi \;$ in $I\}$.   We can rewrite \eqref{eq:VI}
 in the form of
\begin{equation}\label{eq:NeumannVI}
  \int_I \Phi(q-p)dx+\int_I(p'+f)(q'-p')dx\geq 0\qquad\forall\,q
   \in \fK,
\end{equation}
 where
\begin{equation}\label{eq:wDef}
 p=\bar y'
 \end{equation}
 and
 the function $\Phi\in H^1(I)/\R$ is defined by
\begin{equation}\label{eq:DirichletPhi}
  \beta\Phi'=y_d-\bar y.
\end{equation}
\par
 Let the bounded linear functional $L:H^1(I)/\R\longrightarrow\R$ be defined by
\begin{equation}\label{eq:LDef}
  L v=\int_I\Phi v\,dx+\int_I(p'+f)v'dx.
\end{equation}
\par
 Observe that \eqref{eq:NeumannVI} implies
\begin{equation}\label{eq:Support}
  L v=0 \qquad\text{if $v\in H^1(I)/\R$ and $\fA\cap \mathrm{supp}\, v=\emptyset$},
\end{equation}
 since in this case $\pm\epsilon v+p\in\fK$ for $0<\epsilon\ll 1$.
\par
 { Since the active set $\fA$ is a closed subset of $[0,1]$,
 according to \eqref{eq:DirichletActiveSet} there exist two numbers $a,b\in I$ such that
 $a<b$ and $[a,b]\cap \fA=\emptyset$.}
  Let $G=(-1,a)\cup(b,1)$.  Then we have
 (i) $\fA\,\cap I\subset G$
 and
 (ii) there exists  a bounded linear extension operator $\EG:H^1(G)\longrightarrow H^1(I)/\R$.
\begin{remark}\label{rem:Extension}\rm
  Observe that a bounded linear extension operator $E_G^*:H^1(G)\longrightarrow H^1(I)$
  can be constructed by reflections (cf. \cite{ADAMS:2003:Sobolev}).  The operator
  $E_G$ can then be defined by
      $$E_G(v)=E_G^*(v)-\Big(\int_I E_G^*(v)dx\Big)\phi,$$
  where $\phi$ is a smooth function with compact support in $(a,b)$ such that $\int_I \phi\,dx=1$.
\end{remark}
\par
 We define a bounded linear map $\TG:H^1(G)\longrightarrow\R$ by
\begin{equation}\label{eq:TG}
  \TG v= L \tilde v
\end{equation}
 where $\tilde v$ is any function in $H^1(I)/\R$ such that $\tilde v=v$ on $G$.
 $\TG$ is well-defined because
  the existence of $\tilde v$ is guaranteed by the extension operator $\EG$ and the
  independence of the choice of $\tilde v$ follows from  \eqref{eq:Support}.
\par
 Let $v\in H^1(G)$ be nonnegative.  Then $-\epsilon \tilde v+p\in \fK\,$ for
 $0<\epsilon\ll1$ because $p\leq \psi$ on $G$ and $p<\psi$ on the compact set $[a,b]=I\setminus G$.
 Hence we have
\begin{equation}\label{eq:Positivity}
  -\TG v=\epsilon^{-1}\TG (-\epsilon v)=\epsilon^{-1}L(-\epsilon\tilde v)\geq0
\end{equation}
 by \eqref{eq:NeumannVI} and \eqref{eq:LDef}.
\par
 It follows from \eqref{eq:Positivity} and the Riesz-Schwartz Theorem
 \cite{Rudin:1966:RC,Schwartz:1966:Distributions}  for nonnegative functionals
 that
\begin{equation}\label{eq:muG}
  \TG v=-\int_{[-1,a]\cup[b,1]} v\,d\muG \qquad \forall\,v\in H^1(G).
\end{equation}
where $\muG$ is a nonnegative Borel measure on $[-1,a]\cup[b,1]$.
\par
 Because of \eqref{eq:TG} and \eqref{eq:muG}, we have
\begin{equation}\label{eq:LRepresentation}
  -L v=-T(v\big|_G)=\int_{[-1,a]\cup[b,1]}v\,d\muG \qquad \forall\,v\in H^1(I)/\R,
\end{equation}
 and the observation \eqref{eq:Support} implies that $\muG$ is supported on $\fA$.
\par
 We conclude from \eqref{eq:LDef} and \eqref{eq:LRepresentation} that
\begin{equation}\label{eq:KKT}
  \int_I\Phi v\,dx+\int_I(p'+f)v'dx+\int_{[-1,1]} v\,
  d\tilde\mu=0 \qquad\forall\,v\in H^1(I)/\R,
\end{equation}
 where $\tilde\mu$ is the trivial extension of $\muG$ to $[-1,1]$.  It follows that
\begin{equation*}
   \int_{[-1,1]} (\bar y'-\psi)d\mu=0,
\end{equation*}
 where $\mu=\beta\tilde\mu$,
 and in view of  \eqref{eq:Quotient},  \eqref{eq:wDef}, \eqref{eq:DirichletPhi} and \eqref{eq:KKT},
\begin{equation*}
  \int_I(\bar y-y_d)z\,dx+\beta\int_I(\bar y''+f)z''dx+\int_{[-1,1]} z'd\mu=0 \qquad\forall\,z\in V.
\end{equation*}
%
\section{KKT Conditions for the Mixed Boundary Conditions}\label{appendix:Mixed}
 In this case we have, by \eqref{eq:VMixed},
\begin{equation}\label{eq:SD}
  \{y':y\in V\}=\{v\in H^1(I):\,v(1)=0\}=H^1(I;1).
\end{equation}
\par
 Let $\fK=\{v\in H^1(I;1):\,v\leq\psi\quad\text{in $I$}\}$.
 We can rewrite \eqref{eq:VI} in the form of
\begin{equation}\label{eq:PreNeumannVI}
  \int_I \Phi(q-p)dx+\int_I (p'+f)(q'-p')dx\geq 0\qquad\forall\,q
   \in \fK,
\end{equation}
 where
\begin{equation}\label{eq:pDef}
 p=\bar y'\in \fK,
 \end{equation}
 and the function $\Phi\in H^1(I;1)$ is defined by
\begin{equation}\label{eq:Phi}
  \beta\Phi'=y_d-\bar y.
\end{equation}
\par
 Note that $f\in H^1(I)$ by the assumption in \eqref{eq:MixedData}.
 After integration by parts, the inequality \eqref{eq:PreNeumannVI} becomes
\begin{equation}\label{eq:NewNeumannVI}
 - f(-1)[q(-1)-p(-1)]+\int_I (\Phi-f')(q-p)\,dx
 +\int_I p'(q'-p')dx\geq 0\qquad\forall\,q\in\fK.
\end{equation}
 The variational inequality defined by \eqref{eq:pDef} and
 \eqref{eq:NewNeumannVI} is equivalent to a second
 order obstacle problem with mixed boundary conditions
 whose coincidence set is identical to the active set $\fA$ in \eqref{eq:ActiveSet}.
\par
 Since $\psi\in H^2(I)$ by the assumption in \eqref{eq:MixedData},
 we can apply the penalty method in \cite{MS:1972:MixedVI} to show that
\begin{equation}\label{eq:pRegularity}
 \text{the solution $p$ of \eqref{eq:NewNeumannVI} belongs to $H^2(I)$},
\end{equation}
 and, after integration by parts,  we have
\begin{equation}\label{eq:NeumannKKT}
 -f(-1)q(-1)+\int_I (\Phi-f') q\,dx+\int_I p'q'dx
 +\int_{[-1,1]} q\,d\nu=0 \qquad\forall\,q\in H^1(I;1),
\end{equation}
 where
\begin{equation}\label{eq:nuFormula}
  d\nu=(p''+f'-\Phi)dx+(f(-1)+p'(-1))d\delta_{-1},
\end{equation}
 and $\delta_{-1}$ is the Dirac point measure at $-1$.
\par
  The variational inequality
 \eqref{eq:NewNeumannVI} is then equivalent to
\begin{subequations}\label{subeq:Sign}
\begin{alignat}{3}
  p&\leq\psi &\qquad&\text{in}\;I,\label{eq:Sign0}\\
  p''+f'-\Phi&\geq0 &\qquad&\text{in}\;I,\label{eq:Sign1}\\
    f(-1)+p'(-1)&\geq0, \label{eq:Sign2}\\
    \int_{[-1,1]}(p-\psi)d\nu&=0.\label{eq:Complementarity}
\end{alignat}
\end{subequations}
\par
 Consequently the KKT conditions \eqref{eq:KKT1}--\eqref{eq:KKT3} hold
  for the Borel measure
\begin{equation}\label{eq:muFormula}
  \mu=\beta\nu.
\end{equation}
\begin{remark}\label{rem:Specialpsi}\rm
In the special case where $f=0$ and $\psi$ is a positive constant, the condition
   \eqref{eq:Complementarity} implies $p'(-1)=0$ if $-1\not\in\fA$,
   and the conditions \eqref{eq:Sign0} and \eqref{eq:Sign2} imply $p'(-1)=0$ if $-1\in\fA$.
  Therefore we have $p'(-1)=0$ if  $f=0$ and $\psi$ is a positive constant,
  in which case  $\mu$ is absolutely continuous with respect to
  the Lebesgue measure.  Hence it is necessary to choose $\gamma=p'(-1)=0$ in Example~\ref{example:Mixed}.
\end{remark}

\end{document}